\documentclass[12pt,reqno]{amsart}
\usepackage{amsmath, amsthm, amssymb}
\usepackage{hyperref}
\usepackage{enumerate}
\usepackage{url}

\let\OLDthebibliography\thebibliography
\renewcommand\thebibliography[1]{
  \OLDthebibliography{#1}
  \setlength{\parskip}{0pt}
  \setlength{\itemsep}{0pt plus 0.3ex}
}

\usepackage{mathtools}

\usepackage{multirow, array}
\usepackage{placeins}

\usepackage{caption} 
\advance \topmargin by -\headheight
\advance \topmargin by -\headsep
     
\evensidemargin \oddsidemargin
\newtheorem{thm}{Theorem}[section]
\newtheorem{lemma}[thm]{Lemma}

\newtheorem{cor}[thm]{Corollary}

\theoremstyle{definition}

\theoremstyle{remark}

\numberwithin{equation}{section}

\newcommand{\mmod}[1]{\,\,\rm{mod}\,\,#1}

\newcommand*\wrapletters[1]{\wr@pletters#1\@nil}
\def\wr@pletters#1#2\@nil{#1\allowbreak\if&#2&\else\wr@pletters#2\@nil\fi}

\def\alp{{\alpha}} 
\def\bet{{\beta}}  
\def\gam{{\gamma}} 
\def\del{{\delta}}

\def \blam {{\boldsymbol \lam}}
\def\tet{{\theta}}  

\def\lam{{\lambda}} \def\Lam{{\Lambda}}

\def\sig{{\sigma}}

\def\eps{\varepsilon}

\def\le{\leqslant} \def\ge{\geqslant}

\def\d{{\,{\rm d}}}

\def \sig{{\sigma}}

\def \bA {\mathbf A}

\def \bK {\mathbb K}
\def \bN {\mathbb N}

\def \bQ {\mathbb Q}
\def \bR {\mathbb R}
\def \bZ {\mathbb Z}

\def \bL {\mathbf L}

\def \ba {\mathbf a}

\def \bc {\mathbf c}

\def \be {\mathbf e}
\def \bf {\mathbf f}	

\def \bh {\mathbf h}

\def \bk {\mathbf k}

\def \bq {\mathbf q}

\def \bx {\mathbf x}

\def \by {\mathbf y}
\def \bz {\mathbf z}

\def \bzero {\mathbf 0}

\def \btau {\boldsymbol{\tau}}

\def \balp {\boldsymbol{\alp}}
\def \bbet {\boldsymbol{\beta}}

\def \bgam {\boldsymbol{\gam}}
\def \bxi {{\boldsymbol{\xi}}}

\def \fm {\mathfrak m}

\def \fp {\mathfrak p}

\def \ft {\mathfrak t}

\def \fM {\mathfrak M}
\def \fN {\mathfrak N}

\def \fQ {\mathfrak Q}
\def \fR {\mathfrak R}
\def \fS {\mathfrak S}

\def \fU {\mathfrak U}

\def \cF {\mathcal F}

\def \cM {\mathcal M}

\def \cS {\mathcal S}

\def \cU {\mathcal U}

\def \rank {\mathrm{rank}}

\def \dim {\mathrm{dim}}

\def \sinc {\mathrm{sinc}}
\def \meas {\mathrm{meas}}

\def \dag {\dagger}

\begin{document}
\title[Equidistribution on a cubic hypersurface]{Equidistribution of values of linear forms on a cubic hypersurface}
\author[Sam Chow]{Sam Chow}
\address{School of Mathematics, University of Bristol, University Walk, Clifton, Bristol BS8 1TW, United Kingdom}
\email{Sam.Chow@bristol.ac.uk}
\thanks{}
\date{}
\begin{abstract} Let $C$ be a cubic form with rational coefficients in $n$ variables, and let $h$ be the $h$-invariant of $C$. Let $L_1, \ldots, L_r$ be linear forms with real coefficients such that if $\balp \in \mathbb{R}^r \setminus \{ \boldsymbol{0} \}$ then $\boldsymbol{\alp} \cdot \mathbf{L}$ is not a rational form. Assume that $h > 16 + 8 r$. Let $\boldsymbol{\tau} \in \mathbb{R}^r$, and let $\eta$ be a positive real number. We prove an asymptotic formula for the weighted number of integer solutions \mbox{$\mathbf{x} \in [-P,P]^n$} to the system $C(\mathbf{x}) = 0, \: |\mathbf{L}(\mathbf{x}) - \boldsymbol{\tau}| < \eta$. If the coefficients of the linear forms are algebraically independent over the rationals, then we may replace the $h$-invariant condition with the hypothesis $n > 16 +  9 r$, and show that the system has an integer solution. Finally, we show that the values of $\mathbf{L}$ at integer zeros of $C$ are equidistributed modulo one in $\mathbb{R}^r$, requiring only that $h > 16$. 
\end{abstract}
\maketitle

\section{Introduction}
\label{intro}

Recently Sargent \cite{Sar2014} used ergodic methods to establish the equidistribution of values of real linear forms on a rational quadric, subject to modest conditions. His ideas stemmed from quantitative refinements \cite{DM1993, EMM1998} of Margulis' proof \cite{Mar1989} of the Oppenheim conjecture. Such techniques do not readily apply to higher degree hypersurfaces. Our purpose here is to use analytic methods to obtain similar results on a cubic hypersurface. 

Our first theorem is stated in terms of the \emph{h-invariant} of a nontrivial rational cubic form $C$ in $n$ variables, which is defined to be the least positive integer $h$ such that
\begin{equation} \label{hdef}
C(\bx) = A_1(\bx)B_1(\bx) + \ldots + A_h(\bx)B_h(\bx)
\end{equation}
identically, for some rational linear forms $A_1, \ldots, A_h$ and some rational quadratic forms $B_1, \ldots, B_h$. The $h$-invariant describes the geometry of the hypersurface $\{ C=0 \}$, and in fact $n-h$ is the greatest affine dimension of any rational linear space contained in this hypersurface (therefore $1 \le h \le n$).

\begin{thm} \label{thm1} Let $C$ be a cubic form with rational coefficients in $n$ variables, and let $h = h(C)$ be the $h$-invariant of $C$. Let $L_1, \ldots, L_r$ be linear forms with real coefficients in $n$ variables such that if $\balp \in \bR^r \setminus \{ \bzero \}$ then $\balp \cdot \bL$ is not a rational form. Assume that 
\begin{equation} \label{numvars}
h > 16 + 8 r.
\end{equation}
Let $\btau \in \bR^r$ and $\eta > 0$. Let
\begin{equation} \label{wdef}
w(\bx) = \begin{cases}
\exp \Biggl(- \displaystyle \sum_{j \le n} \frac1 {1- x_j^2} \Biggr), & |\bx| < 1 \\
0, & |\bx| \ge 1,
\end{cases}
\end{equation}
and define the weighted counting function
\[
N_w(P) = \sum_{\substack{\bx \in \bZ^n: \\ C(\bx) = 0, \: |\bL(\bx) - \btau| < \eta}} w(\bx / P).
\]
Then
\begin{equation} \label{asymp}
N_w(P) = (2 \eta)^r \fS \chi_w P^{n-r-3} + o(P^{n-r-3})
\end{equation}
as $P \to \infty$, where
\begin{equation} \label{SS}
\fS = \sum_{q \in \bN} q^{-n} \sum_{\substack{a \mmod q \\ (a,q) = 1}} \:
\sum_{\bx \mmod q} e_q(aC(\bx))
\end{equation}
and
\begin{equation} \label{chidef}
\chi_w = \int_{\bR^{r+1}} \int_{\bR^n} w(\bx)  e(\bet_0 C(\bx) + \balp \cdot \bL(\bx)) \d \bx \d \bet_0 \d \balp.
\end{equation}
Further, we have $\fS \chi_w > 0$.
\end{thm}

We interpret $N_w(P)$ as a weighted count for the number of integer solutions $\bx \in (-P,P)^n$ to the system
\begin{equation} \label{system}
C(\bx) = 0, \quad |\bL(\bx) - \btau| < \eta.
\end{equation}
The smooth weight function $w(\bx)$ defined in \eqref{wdef} is taken from \cite{HB1996}. Importantly, it has bounded support and bounded partial derivatives of all orders. One advantage of the weighted approach is that it enables the use of Poisson summation.

The \emph{singular series} $\fS$ may be interpreted as a product of $p$-adic densities of points on the hypersurface $\{ C=0 \}$; see \cite[\S 7]{Bir1962}. Note that this captures the arithmetic of $C$, but that no such arithmetic is present for the linear forms $L_1, \ldots, L_r$, since they are `irrational' in a precise sense. 

The weighted \emph{singular integral} $\chi_w$ arises naturally in our proof as the right hand side of \eqref{chidef}. Using Schmidt's work \cite{Sch1982b, Sch1985}, we can interpret $\chi_w$ as the weighted real density of points on the variety $\{ C = L_1 = \ldots = L_r = 0 \}$. For $L > 0$ and $\xi \in \bR$, let
\[
\psi_L(\xi) = L \cdot \max(0,1 - L|\xi|).
\]
For $\bxi \in \bR^{r+1}$, put
\[
\Psi_L(\bxi) = \prod_{v \le r+1} \psi_L(\xi_v).
\]
With $\bf = (C,\bL)$, set
\[
I_L(\bf) = \int_{\bR^n} w(\bx) \Psi_L(\bf(\bx)) \d \bx,
\]
and define
\begin{equation} \label{cSchmidt}
\chi_w = \lim_{L \to \infty} I_L(\bf).
\end{equation}
We shall see that the limit \eqref{cSchmidt} exists, and that this definition is equivalent to the analytic definition \eqref{chidef}.

We may replace the condition on the $h$-invariant by a condition on the number of variables, at the expense of assuming that the coefficients of the linear forms are in `general position'.

\begin{thm} \label{thm2} Let $C$ be a cubic form with rational coefficients in $n$ variables. Let $L_1, \ldots, L_r$ be linear forms in $n$ variables, with real coefficients that are algebraically independent over $\bQ$. Assume that
\[
n > 16 + 9 r.
\]
Let $\btau \in \bR^r$ and $\eta > 0$. Then there exists $\bx \in \bZ^n$ satisfying \eqref{system}.
\end{thm}

The point of this work is to show that the zeros of a rational cubic form are, in a strong sense, well distributed. Similar methods may be applied if $C$ is replaced by a higher degree form; the simplest results would concern non-singular forms of odd degree. The unweighted analogue of the case $r = 0$ of Theorem \ref{thm1} has been solved, assuming only that \mbox{$h \ge 16$}; see remark (B) in the introduction of Schmidt's paper \cite{Sch1985}. For the case \mbox{$r = 0$} of Theorem \ref{thm2}, we can choose $\bx = \bzero$, or note from \cite{HB2007} that fourteen variables suffice to ensure a nontrivial solution. We shall assume throughout that $r \ge 1$.

The fact that we have linear inequalities rather than equations does genuinely increase the difficulty of the problem. For example, suppose we wished to nontrivially solve the system of equations $C = L_1 = \ldots = L_r = 0$, where here the $L_i$ are linear forms with rational coefficients; assume for simplicity that the $L_i$ are linearly independent. Using the linear equations, we could determine $r$ of the variables in terms of the remaining $n-r$ variables, and substituting into $C(\bx) = 0$ would yield a homogeneous cubic equation in $n-r$ variables. Thus, by Heath-Brown's result \cite{HB2007}, we could solve the system given \mbox{$n \ge 14 + r$} variables.

We use the work of Browning, Dietmann and Heath-Brown \cite{BDHB2014} as a benchmark for comparison. Those authors investigate simultaneous rational solutions to one cubic equation $C=0$ and one quadratic equation $Q=0$. They establish the smooth Hasse principle under the assumption that
\[
\min(h(C), \rank(Q)) \ge 37.
\]
We expect to do somewhat better when considering one cubic equation and one linear inequality simultaneously, and we do. Substituting $r=1$ into \eqref{numvars}, we see that we only require $h(C) > 24$.

To prove Theorem \ref{thm1}, we use the Hardy--Littlewood method \cite{Vau1997} in unison with Freeman's variant \cite{Fre2002} of the Davenport--Heilbronn method \cite{DH1946}. The central objects to study are the weighted exponential sums
\[
S(\alp_0, \balp) = \sum_{\bx \in \bZ^n} w(\bx / P) e(\alp_0 C(\bx) + \balp \cdot \bL(\bx)).
\]
If $|S(\alp_0, \balp)|$ is substantially smaller than the trivial estimate $O(P^n)$, then we may adapt \cite[Lemma 4]{DL1964} to rationally approximate $\alp_0$ (see \S \ref{OneRational}). 

In \S \ref{Poisson}, we use Poisson summation to approximately decompose our exponential sum into archimedean and nonarchimedean components. In \S \ref{MoreRational}, we use Heath-Brown's first derivative bound \cite[Lemma 10]{HB1996} and a classical pruning argument \cite[Lemma 15.1]{Dav2005} to essentially obtain good simultaneous rational approximations to $\alp_0$ and $\balp$. In \S \ref{MVE}, we combine classical ideas with Heath-Brown's first derivative bound to obtain a mean value estimate of the correct order of magnitude. In \S \ref{DHmethod}, we define our Davenport--Heilbronn arcs, and in particular use the methods of Bentkus, G\"otze and Freeman \cite{BG1999,Fre2002,Woo2003} to obtain nontrivial cancellation on the minor arcs, thereby establishing the asymptotic formula \eqref{asymp}. We complete the proof of Theorem \ref{thm1} in \S \ref{positivity} by explaining why $\fS$ and $\chi_w$ are positive. It is then that we justify the interpretation of $\chi_w$ as a weighted real density.

We prove Theorem \ref{thm2} in \S \ref{general}. By Theorem \ref{thm1}, it suffices to consider the case where the $h$-invariant is not too large. With $A_1, \ldots, A_h$ as in \eqref{hdef}, we solve the system \eqref{system} by solving the linear system
\[
\bA(\bx) = \bzero, \qquad |\bL(\bx) - \btau| < \eta.
\]
Since the $h$-invariant is not too large, this system has more variables than constraints, and can be solved using methods from linear algebra and diophantine approximation, providing that the coefficients of $L_1, \ldots, L_r$ are algebraically independent over $\bQ$.

In \S \ref{equidistribution}, we shall prove the following equidistribution result.
\begin{thm} \label{thm3}
Let $C$ be a cubic form with rational coefficients in $n$ variables, let $h = h(C)$ be the $h$-invariant of $C$, and assume that $h > 16$. Let $r \in \bN$, and let $L_1, \ldots, L_r$ be linear forms with real coefficients in $n$ variables such that if $\balp \in \bR^r \setminus \{ \bzero \}$ then $\balp \cdot \bL$ is not a rational form. Let
\[
Z = \{ \bx \in \bZ^n: C(\bx) = 0 \},
\]
and order this set by height $|\bx|$. Then the values of $\bL(Z)$ are equidistributed modulo one in $\bR^r$.
\end{thm}
A little surprisingly, perhaps, we do not require $h$ to grow with $r$. By a multidimensional Weyl criterion, it will suffice to investigate $S_u(\alp_0, \bk)$ for a fixed nonzero integer vector $\bk$, where
\[
S_u(\alp_0, \balp) = \sum_{|\bx| < P} e(\alp_0C(\bx) + \balp \cdot \bL(\bx))
\]
is the unweighted analogue of $S(\alp_0, \balp)$. A simplification of the method employed to prove Theorem \ref{thm1} will complete the argument. 

Rather than using the $h$-invariant, one could instead consider the dimension of the \emph{singular locus} of the affine variety $\{ C = 0 \}$, as in Birch's paper \cite{Bir1962}. Such an analysis would imply results for arbitrary nonsingular cubic forms in sufficiently many variables. These types of theorems are discussed in \S \ref{nonsingular}.

We adopt the convention that $\eps$ denotes an arbitrarily small positive number, so its value may differ between instances. For $x \in \bR$ and $q \in \bN$, we put $e(x) = e^{2 \pi i x}$ and $e_q(x) = e^{2 \pi i x / q}$. Bold face will be used for vectors, for instance we shall abbreviate $(x_1,\ldots,x_n)$ to $\bx$, and define $|\bx| = \max(|x_1|, \ldots, |x_n|)$. We will use the unnormalised sinc function, given by $\sinc(x) = \sin(x)/x$ for $x \in \bR \setminus \{0\}$ and $\sinc(0) = 1$. For $x \in \bR$, we write $\| x \|$ for the distance from $x$ to the nearest integer.

We regard $\btau$ and $\eta$ as constants. The word \emph{large} shall mean in terms of $C, \bL, \eps$ and constants, together with any explicitly stated dependence. Similarly, the implicit constants in Vinogradov's and Landau's notation may depend on $C, \bL, \eps$ and constants, and any other dependence will be made explicit. The pronumeral $P$ denotes a large positive real number. The word \emph{small} will mean in terms of $C, \bL$ and constants. We sometimes use such language informally, for the sake of motivation; we make this distinction using quotation marks. 

The author is very grateful towards Trevor Wooley for his energetic supervision.

\section{One rational approximation}
\label{OneRational}

First we use Freeman's kernel functions \cite[\S 2.1]{Fre2002} to relate $N_w(P)$ to our exponential sums $S(\alp_0,\balp)$. We shall define 
\[
T: [1, \infty) \to [1, \infty)
\]
in due course. For now, it suffices to note that 
\begin{equation} \label{Tbound}
T(P) \le P, 
\end{equation}
and that $T(P) \to \infty$ as $P \to \infty$. Put
\begin{equation} \label{Ldef}
L(P) = \max(1,\log T(P)), \qquad \rho = \eta L(P)^{-1}
\end{equation}
and
\begin{equation} \label{Kdef}
K_{\pm}(\alp) = \frac {\sin(\pi \alp \rho) \sin(\pi \alp(2 \eta \pm \rho))} {\pi^2 \alp^2 \rho}. 
\end{equation}
From \cite[Lemma 1]{Fre2002} and its proof, we have
\begin{equation} \label{Kbounds}
K_\pm(\alp) \ll \min(1, L(P) |\alp|^{-2})
\end{equation}
and
\begin{equation} \label{Ubounds}
0 \le \int_\bR e(\alp t) K_{-}(\alp)\d\alp \le U_\eta(t) \le \int_\bR e(\alp t) K_{+}(\alp)\d\alp \le 1,
\end{equation}
where
\begin{equation*}
U_\eta(t) = 
\begin{cases}
1, &\text{if } |t| < \eta \\
0, &\text{if } |t| \ge \eta.
\end{cases}
\end{equation*}

For $\balp \in \bR^r$, write
\begin{equation} \label{Kprod}
\bK_\pm(\balp) = \prod_{k \le r} K_\pm(\alp_k).
\end{equation}
Let $U$ be a unit interval, to be specified later. By rescaling, we may assume that $C$ has integer coefficients. The inequalities \eqref{Ubounds} and the identity
\begin{equation} \label{orth}
\int_U e(\alp m) \d \alp = \begin{cases}
1, &\text{if } m = 0 \\
0, &\text{if } m \in \bZ \setminus \{ 0 \}
\end{cases}
\end{equation}
now give
\[
R_{-}(P)  \le N_w(P) \le R_+(P), 
\]
where 
\[ 
R_\pm(P) = \int_{\bR^r} \int_U S(\alp_0,\balp) e(-\balp \cdot \btau) \bK_\pm(\balp) \d \alp_0 \d \balp. 
\]
In order to prove \eqref{asymp}, it therefore remains to show that
\begin{equation} \label{goal1}
R_\pm(P) = (2 \eta)^r \fS \chi_w P^{n-r-3} + o(P^{n-r-3}).
\end{equation}

We shall in fact need to investigate the more general exponential sum
\[
g(\alp_0, \blam) = \sum_{\bx \in \bZ^n} w(\bx / P) e(\alp_0 C(\bx) + \blam \cdot \bx).
\]
We note at once that
\[
S(\alp_0,\balp) = g(\alp_0, \Lam \balp),
\]
where
\begin{equation} \label{lamdef}
L_i(\bx) = \lam_{i,1} x_1 + \ldots + \lam_{i,n} x_n \qquad (1 \le i \le r)
\end{equation}
and
\begin{equation} \label{LamDef}
\Lam = \begin{pmatrix}
\lam_{1,1} & \ldots & \lam_{r,1} \\
\vdots &  & \vdots \\
\lam_{1,n} & \ldots & \lam_{r,n}
\end{pmatrix}.
\end{equation}

Fix a large positive constant $C_1$.

\begin{lemma} \label{DL} If $0 < \tet < 1$ and 
\[
|g(\alp_0, \blam)| \ge P^{n- \frac h4 \tet + \eps}
\]
then there exist relatively prime integers $q$ and $a$ satisfying
\begin{equation} \label{DLapprox}
1 \le q \le C_1 P^{2 \tet}, \qquad |q \alp_0 - a| < P^{2 \tet - 3}.
\end{equation}
The same is true if we replace $g(\alp_0, \blam)$ by 
\begin{equation} \label{gudef}
g_u(\alp_0,\blam) := \sum_{|\bx| < P} e(\alp_0 C(\bx) + \blam \cdot \bx),
\end{equation}
or by
\[
\sum_{1 \le x_1, \ldots, x_n \le P} e(\alp_0C(\bx) + \blam \cdot \bx).
\]
\end{lemma}

\begin{proof} 
Our existence statement is a weighted analogue of \cite[Lemma 4]{DL1964}. One can follow \cite[\S 3]{DL1964}, \emph{mutatis mutandis}. The only change required is in proving the analogue of \cite[Lemma 1]{DL1964}. Weights are introduced into the linear exponential sums that arise from Weyl differencing, but these weights are easily handled using partial summation. Our final statement holds with the same proof: one imitates \cite[\S 3]{DL1964}.
\end{proof}

For $\tet \in (0,1)$, $q \in \bN$ and $a \in \bZ$, let $\fN_{q,a}(\tet)$ be the set of $(\alp_0,\balp) \in U \times \bR^r$ satisfying \eqref{DLapprox}, and let $\fN(\tet)$ be the union of the sets $\fN_{q,a}(\tet)$ over relatively prime $q$ and $a$. This union is disjoint if $\tet < 3/4$. Indeed, suppose we have \eqref{DLapprox} for some relatively prime integers $q$ and $a$, and that we also have relatively prime integers $q'$ and $a'$ satisfying
\[
1 \le q' \ll P^{2\tet}, \qquad |q' \alp_0 - a'| < P^{2\tet-3}.
\]
The triangle inequality then yields
\[
|a/q - a'/q'| < P^{2 \tet - 3} (1/q + 1/q') < 1/(qq').
\]
Hence $a/q = a'/q'$, so $a'=a$ and $q'=q$.

We prune our arcs using the well known procedure in \cite[Lemma 15.1]{Dav2005}. Fix a small positive real number $\del$. The following corollary shows that we may restrict attention to $\fN(1/2 - \del)$. 

\begin{cor} \label{FirstArcs} We have
\[
\int_{U \times \bR^r \setminus \fN(1/2 - \del)} |S(\alp_0, \balp) \bK_\pm (\balp)| \d \alp_0 \d \balp = o(P^{n-r-3}). 
\]
\end{cor}

\begin{proof}
Choose real numbers $\psi_1, \ldots, \psi_{t-1}$ such that
\[
1/2 - \del = \psi_0 < \psi_1 < \ldots < \psi_{t-1} < \psi_t = 0.8.
\]
Dirichlet's approximation theorem \cite[Lemma 2.1]{Vau1997} implies that \mbox{$\fN(\tet_t) = U \times \bR^r$.} Let $\fU$ be an arbitrary unit hypercube in $r$ dimensions, and put $\cU = U \times \fU$. Since
\[
\meas(\fN(\tet) \cap \cU) \ll P^{4 \tet-3},
\]
Lemma \ref{DL} gives
\[
\int_{(\fN(\psi_g) \setminus \fN(\psi_{g-1})) \cap \cU} |S(\alp_0,\balp)| \d \alp_0 \d \balp 
\ll   P^{4 \psi_g - 3 + n - h \psi_{g-1}/4 + \eps} \qquad (1 \le g \le t).
\]
This is $O(P^{n-r-3-\eps})$ if $\psi_{g-1} / \psi_g \simeq 1$, since $\psi_{g-1} \ge 1/2 - \del$ and $h \ge 17 + 8r$. Thus, we can choose $\psi_1, \ldots, \psi_{t-1}$ with $t \ll 1$ satisfactorily, to ensure that
\[
\int_{\cU \setminus \fN(1/2 - \del)} |S(\alp_0,\balp)| \d \alp_0 \d \balp \ll P^{n-r-3-\eps}.
\]
The desired inequality now follows from \eqref{Tbound}, \eqref{Ldef}, \eqref{Kbounds} and \eqref{Kprod}. 
\end{proof}

Thus, to prove \eqref{goal1} and hence \eqref{asymp}, it remains to show that
\begin{equation} \label{goal2}
\int_{\fN(1/2-\del)} S(\alp_0,\balp) e(-\balp \cdot \btau) \bK_\pm(\balp) \d \alp_0 \d \balp =  (2 \eta)^r \fS \chi_w P^{n-r-3} + o(P^{n-r-3}).
\end{equation}

\section{Poisson summation}
\label{Poisson}

Put
\begin{equation} \label{put}
\alp_0 = \frac a q + \bet_0, \qquad \blam = q^{-1} \ba + \bbet,
\end{equation}
where $q \ge 1$ and $a$ are relatively prime integers, and where $\ba \in \bZ^n$. By periodicity, we have
\[
g(\alp_0,\blam) = \sum_{\by \mmod q} e_q (aC(\by) + \ba \cdot \by) I_\by (q,\bet_0,\bbet),
\]
where
\[
I_\by(q,\bet_0,\bbet) = \sum_{\bx \equiv \by \mmod q} w(\bx/P) e(\bet_0 C(\bx) + \bbet \cdot \bx).
\]
Since
\[
I_\by(q,\bet_0,\bbet) = \sum_{\bz \in \bZ^n} w \Bigl(\frac{\by + q \bz}P \Bigr) 
e(\bet_0 C(\by + q \bz) + \bbet \cdot (\by + q \bz)),
\]
Poisson summation yields
\[
I_\by(q,\bet_0,\bbet) = \sum_{\bc \in \bZ^n} \int_{\bR^n} w \Bigl(\frac{\by + q \bz}P \Bigr) 
e(\bet_0 C(\by + q \bz) + \bbet \cdot (\by + q \bz) - \bc \cdot \bz) \d\bz.
\]
Changing variables now gives
\[
I_\by(q,\bet_0,\bbet) = (P/q)^n \sum_{\bc \in \bZ^n} e_q(\bc \cdot \by) 
I(P^3 \bet_0, P(\bbet - \bc / q)),
\]
where
\begin{equation} \label{intdef}
I(\gam_0, \bgam) = 
\int_{\bR^n} w (\bx) e(\gam_0 C(\bx) + \bgam \cdot \bx) \d \bx.
\end{equation}

Write 
\begin{equation} \label{sumdef}
S_{q,a,\ba} = \sum_{\by \mmod q} e_q (aC(\by) + \ba \cdot \by),
\end{equation}
and let
\[
g_0(\alp_0, \blam) = (P/q)^n S_{q,a,\ba} I(P^3 \bet_0, P\bbet)
\]
be the $\bc = \bzero$ contribution to $g(\alp_0, \blam)$. Then
\begin{equation} \label{PS1}
g(\alp_0, \blam) - g_0(\alp_0, \blam)
= (P/q)^n \sum_{\bc \ne \bzero} S_{q,a,\ba + \bc} I(P^3 \bet_0, P(\bbet - \bc / q)).
\end{equation}
We shall bound the right hand side from above, in the case where we have \eqref{DLapprox} and $|\bbet| \le 1/(2q)$. For future reference, we note that specialising \mbox{$(q,a,\ba) = (1,0,\bzero)$} in \eqref{PS1} gives
\begin{equation} \label{PS2}
g(\alp_0, \blam) - P^n I(P^3 \alp_0, P \blam) = P^n \sum_{\bc \ne \bzero} I(P^3 \alp_0, P\blam - P\bc).
\end{equation}

We bound $S_{q,a,\ba}$ by imitating \cite[Lemma 15.3]{Dav2005}.

\begin{lemma} \label{Sbound}
Let $q \ge 1$ and $a$ be relatively prime integers, and let $\psi > 0$. Then
\begin{equation} \label{Sineq}
S_{q,a,\ba} \ll_\psi q^{n-h/8+\psi}.
\end{equation}
\end{lemma}

\begin{proof} Suppose for a contradiction that $q$ is large in terms of $\psi$ and
\[
|S_{q,a,\ba}| > q^{n-h/8+\psi}.
\]
Recall that 
\[
S_{q,a,\ba} = \sum_{1 \le y_1, \ldots, y_n \le q} e_q(aC(\by) + \ba \cdot \by).
\]
We may assume without loss that $\psi < 1$. By Lemma \ref{DL}, with $P=q$ and $\tet = 1/2 - \psi/n$, there exist $s,b \in \bZ$ such that 
\[
1 \le s < q, \qquad |sa/q - b| < q^{-1}.
\]
Now $b/s = a/q$, which is impossible because $(a,q) = 1$ and $1 \le s < q$.
\end{proof}

Let $\bc \in \bZ^n \setminus \{ \bzero \}$, and suppose we have \eqref{DLapprox}, for some $\tet \in (0, 1/2 - \del]$ and some relatively prime $q,a \in \bZ$. Define $a_j$ by rounding $q \lam_j$ to the nearest integer, rounding down if $q \lam_j$ is half of an odd integer ($1 \le j \le n$). Since $|\bc|/q \ge 1/q \ge 2|\bbet|$, we have
\[
|P(\bbet - \bc/q)| \gg P|\bc| / q \gg P/q \gg P^{3+2 \del} |\bet_0|.
\]
Now \cite[Lemma 10]{HB1996} gives
\begin{equation} \label{Ibound0}
I(P^3 \bet_0, P(\bbet - \bc / q)) \ll (P|\bc|/q)^{-n-\eps}.
\end{equation}
By \eqref{PS1}, \eqref{Sineq} and \eqref{Ibound0}, we have
\begin{equation} \label{PS3}
g(\alp_0,\blam) - g_0(\alp_0,\blam) \ll P^{-\eps} q^{n+2\eps - h/8} \ll q^{n - h/8 + \eps}.
\end{equation}

Let $\fU$ be an arbitrary unit hypercube in $r$ dimensions, and put $\cU = U \times \fU$. With $\tet = 1/2 - \del$, we now have
\[
\int_{\fN(\tet) \cap \cU} |S(\alp_0,\balp) - S_0(\alp_0,\balp)| \d \alp_0 \d \balp
\ll \sum_{q \le C_1 P^{2\tet}} q^{n - h/8 + \eps} P^{2\tet - 3},
\]
where for $(\alp_0,\balp) \in \fN_{q,a}(\tet)$ with $(a,q)=1$ we have written 
\[
S_0(\alp_0,\balp) = g_0(\alp_0, \Lam \balp).
\]
Hence
\[
\int_{\fN(\tet) \cap \cU} |S(\alp_0,\balp) - S_0(\alp_0,\balp)| \d \alp_0 \d \balp
\ll P^{4\tet - 3} (P^{2\tet})^{n-h/8 + \eps} \ll P^{n-r-3-\eps},
\]
since $h \ge 17 + 8r$. The bounds \eqref{Tbound}, \eqref{Ldef}, \eqref{Kbounds} and \eqref{Kprod} now yield
\[
\int_{\fN(\tet)} |S(\alp_0,\balp) - S_0(\alp_0,\balp)| \cdot |\bK_\pm(\balp)| \d \alp_0 \d \balp = o(P^{n-3}).
\]
Thus, to prove \eqref{goal2} and hence \eqref{asymp}, it suffices to show that
\begin{equation} \label{goal3}
\int_{\fN(1/2-\del)} S_0(\alp_0,\balp) e(-\balp \cdot \btau) \bK_\pm(\balp) \d \alp_0 \d \balp 
= (2 \eta)^r \fS \chi_w P^{n-r-3} + o(P^{n-r-3}).
\end{equation}

\section{More rational approximations}
\label{MoreRational}

For $\tet \in (0,1/2]$ and integers $q,a,a_1,\ldots,a_n$, let $\fR_{q,a,\ba}(\tet)$ denote the set of $(\alp_0,\balp) \in U \times \bR^r$ satisfying
\begin{equation} \label{R1}
|q\alp_0 - a| < P^{2 \tet - 3}, \qquad |q \Lam \balp - \ba| < P^{(2+\del) \tet - 1},
\end{equation}
and let $\fR(\tet)$ be the union of the sets $\fR_{q,a,\ba}(\tet)$ over integers $q,a,a_1,\ldots,a_n$ satisfying
\begin{equation} \label{R2}
1 \le q \le C_1 P^{2\tet}, \qquad (a,q) = 1.
\end{equation}
Note that this union is disjoint if $\tet < (2+\del)^{-1}$. Let $\fU$ be an arbitrary unit hypercube in $r$ dimensions, and put $\cU = U \times \fU$. Then
\[
\meas(\fR(\tet) \cap \cU) \ll P^{4 \tet - 3 - r + (2+\del) \tet r},
\]
since our hypothesis on $\bL$ implies that $\Lam$ has $r$ linearly independent rows.

Fix a small positive real number $\tet_0$. The following lemma shows that we may restrict attention to $\fR(\tet_0)$.

\begin{lemma} We have
\[
\int_{\fN(1/2-\del) \setminus \fR(\tet_0)} |S_0(\alp_0,\balp) \bK_\pm(\balp)| \d \alp_0 \d \balp = o(P^{n-r-3}).
\]
\end{lemma}

\begin{proof}
Note that $\fN(1/2-\del) \subseteq \fN(1/2) = \fR(1/2)$. Let 
\[
(\alp_0,\balp) \in \fN(1/2 - \del) \cap \fR(\tet_g) \setminus \fR(\tet_{g-1})
\]
for some $g \in \{1,2,\ldots,t\}$, where
\[
0 < \tet_0 < \tet_1 < \ldots < \tet_t = 1/2.
\]

First suppose that $|S(\alp_0,\balp)| \ge P^{n-h \tet_{g-1}/4 + \eps}$. By Lemma \ref{DL}, there exist relatively prime integers $q$ and $a$ satisfying
\[
1 \le q \le C_1 P^{2 \tet_{g-1}}, \qquad |q \alp_0 - a| < P^{2 \tet_{g-1} - 3}.
\]
Let $\bet_0, \ba$ and $\bbet$ be as in \S \ref{Poisson}, with $\blam = \Lam \balp$. Since $(\alp_0,\balp) \notin \fR(\tet_{g-1})$, we must have $q |\bbet| \ge P^{(2+\del)\tet_{g-1} - 1}$. Now
\[
P|\bbet| \gg q^{-1} P^{(2+\del) \tet_{g-1}} \gg P^{\del \tet_{g-1}} P^3|\bet_0|,
\]
so \cite[Lemma 10]{HB1996} yields
\[
I(P^3 \bet_0, P \bbet) \ll_N (q^{-1} P^{(2+\del) \tet_{g-1}} )^{-N} \ll P^{-N \del \tet_{g-1}}
\]
for any $N > 0$. Choosing $N$ large now gives $S_0(\alp_0, \balp) \ll 1$.

Now suppose instead that $|S(\alp_0,\balp)| < P^{n-h/4 \tet_{g-1} + \eps}$. As \mbox{$(\alp_0,\balp) \in \fN(1/2 - \del)$}, there exist relatively prime integers $q$ and $a$ such that \mbox{$(\alp_0,\balp) \in \fN_{q,a}(1/2 - \del)$}. From \eqref{PS3}, we have
\[
S(\alp_0,\balp) - S_0(\alp_0,\balp) \ll q^{n-h/8 + \eps} \ll P^{n-h/8},
\]
and now the triangle inequality yields
\begin{equation} \label{S0bound}
S_0(\alp_0,\balp) \ll \max(P^{n-h \tet_{g-1} /4 + \eps},  P^{n-h/8}) = P^{n-h \tet_{g-1} /4 + \eps}.
\end{equation}

The bound \eqref{S0bound} is valid in both cases, so
\[
\int |S_0(\alp_0,\balp)| \d \alp_0\d \balp
\ll  P^{4 \tet_g - 3 - r + (2+\del) \tet_g r} P^{n-h \tet_{g-1} /4 + \eps},
\]
where the integral is over $\fN(1/2-\del) \cap \cU \cap  \fR(\tet_g) \setminus \fR(\tet_{g-1})$. The right hand side is $O(P^{n-r-3-\eps})$ if $\frac{\tet_g}{\tet_{g-1}} \simeq 1$, since $h \ge 17 + 8r$. We can therefore choose $\tet_1, \ldots, \tet_{t-1}$ with $t \ll 1$ satisfactorily, to ensure that
\[
\int_{\cU \cap \fN(1/2-\del) \setminus \fR(\tet_0)} |S_0(\alp_0,\balp)| \d \alp_0 \d \balp
\ll P^{n-r-3-\eps}.
\]
The desired inequality now follows from \eqref{Tbound}, \eqref{Ldef}, \eqref{Kbounds} and \eqref{Kprod}. 
\end{proof}

Thus, to prove \eqref{goal3} and hence \eqref{asymp}, it suffices to show that
\begin{equation} \label{goal4}
\int_\fR S_0(\alp_0,\balp) e(-\balp \cdot \btau) \bK_\pm(\balp) \d \alp_0 \d \balp 
= (2 \eta)^r \fS \chi_w P^{n-r-3} + o(P^{n-r-3})
\end{equation}
where, now and henceforth, we write
\[
\fR = \fR_P = \fR(\tet_0), \qquad \fR(q,a,\ba) = \fR_{q,a,\ba}(\tet_0).
\]
We now choose our unit interval
\begin{equation} \label{interval}
U = (P^{2 \tet_0 - 3}, 1 + P^{2 \tet_0 - 3}].
\end{equation}
This choice ensures that if the conditions \eqref{R1} and \eqref{R2} hold with $\tet = \tet_0$ for some $(\alp_0,\balp) \in \bR \times \bR^r$ then $\alp_0 \in U$ if and only if $1 \le a \le q$. In particular, the set $\fR$ is the disjoint union of the sets $\fR(q,a,\ba)$ over integers $q,a,a_1,\ldots,a_n$ satisfying 
\begin{equation} \label{R2b}
1 \le a \le q \le C_1 P^{2\tet_0}, \qquad (a,q) = 1.
\end{equation}

\section{A mean value estimate}
\label{MVE}

We begin by bounding $I(\gam_0,\bgam)$. In light of \eqref{PS2}, the first step is to bound $g(\alp_0, \blam)$.

\begin{lemma} \label{gbound} 
Let $\xi$ be a small positive real number. Let $\alp_0 \in \bR$ and $\blam \in \bR^n$ with $|\alp_0| < P^{-3/2}$. Then
\[
g(\alp_0, \blam) \ll P^{n+\xi} (P^3 |\alp_0|)^{-h/8}
\]
and
\[
g_u(\alp_0, \blam) \ll P^{n+\xi} (P^3 |\alp_0|)^{-h/8},
\]
where $g_u(\alp_0, \blam)$ is given by \eqref{gudef}.
\end{lemma}

\begin{proof}
This follows from the argument of the corollary to \cite[Lemma 4.3]{Bir1962}, using Lemma \ref{DL}.
\end{proof}

\begin{lemma} \label{Ibound}
We have
\begin{equation} \label{Iineq}
I(\gam_0,\bgam) \ll \frac1 {1 + (|\gam_0| + |\bgam|)^{h/8-\eps}}.
\end{equation}
\end{lemma}

\begin{proof}
As $I(\gam_0,\bgam) \ll 1$, we may assume that $|\gam_0| + |\bgam|$ is large. From \eqref{PS2}, we have
\[
I(\gam_0,\bgam) = P^{-n} g(\gam_0 / P^3, \bgam / P) - \sum_{\bc \ne \bzero} I(\gam_0, \bgam - P \bc).
\]
Since $I(\gam_0,\bgam)$ is independent of $P$, we are free to choose $P = (|\gam_0| + |\bgam|)^n$. By Lemma \ref{gbound} and \cite[Lemma 10]{HB1996}, we now have
\begin{equation} \label{Iprelim}
I(\gam_0,\bgam) \ll P^{\eps/n} |\gam_0|^{-h/8} = \frac{ (|\gam_0| + |\bgam|)^\eps} { |\gam_0|^{h/8}}.
\end{equation}

Let $C_2$ be a large positive constant. If $|\bgam| \ge C_2 |\gam_0|$ then \cite[Lemma 10]{HB1996} yields $I(\gam_0, \bgam) \ll_N |\bgam|^{-N} \ll_N (|\gam_0| + |\bgam|)^{-N}$ for any $N > 0$, while if $|\bgam| < C_2 |\gam_0|$ then \eqref{Iprelim} gives
\[
I(\gam_0, \bgam) \ll (|\gam_0| + |\bgam|)^{\eps - h/8}.
\]
The latter bound is valid in either case. As $|\gam_0| + |\bgam|$ is large, our proof is complete.
\end{proof}

We now have all of the necessary ingredients to obtain a mean value estimate of the correct order of magnitude. Let $\fU$ be an arbitrary unit hypercube in $r$ dimensions, and put $\cU = U \times \fU$. Let $V \subseteq \{1,2,\ldots,n\}$ index $r$ linearly independent rows of $\Lam$. When $(\alp_0,\balp) \in \fR(q,a,\ba)$ and $(a,q) = 1$, write
\[
F(\alp_0, \balp) = F(\alp_0, \balp; P) = \prod_{v \le n} (q+ P | q \lam_v - a_v|)^{-1}, 
\]
where $\blam = \Lam \balp$. As $h/8 > r+2$, Lemmas \ref{Sbound} and \ref{Ibound} imply that
\begin{align*}
&S_0(\alp_0, \balp) F(\alp_0,\balp)^{-\eps} \\
&\ll P^n q^{-r-2-\eps} (1+P^3|\alp_0 - a/q|)^{-1-\eps} 
\prod_{v \in  V} (1+ P|\lam_v - a_v/q|)^{-1-\eps}.
\end{align*}
For each $q$, we can choose from $O(q^{r+1})$ values of $a, a_v$ ($v \in V$) for which $\fR(q,a,\ba) \cap \cU$ is nonempty. Thus, an invertible change of variables gives
\begin{align}
\notag &\int_{\fR \cap \cU} |S_0(\alp_0,\balp)| F(\alp_0, \balp)^{-\eps} \d \alp_0 \d \balp \\
\notag &\ll P^n \sum_{q \in \bN} q^{-1-\eps} \int_\bR (1+P^3|\bet_0|)^{-1-\eps} \d \bet_0 \cdot
\int_{\bR^r} \prod_{j \le r} (1+P|\alp'_j|)^{-1-\eps} \d \balp' \\
\label{MV} &\ll P^{n-r-3}.
\end{align}
Positivity has permitted us to complete the summation and the integrals to infinity for an upper bound.

\section{The Davenport--Heilbronn method}
\label{DHmethod}

In this section, we specify our Davenport--Heilbronn dissection, and complete the proof of \eqref{asymp}. The bound \eqref{MV} will suffice on the Davenport--Heilbronn major and trivial arcs, but on the minor arcs we shall need to bound $F(\alp_0,\balp)$ nontrivially. Using the methods of Bentkus, G\"otze and Freeman, as exposited in \cite[Lemmas 2.2 and 2.3]{Woo2003}, we will show that $F(\alp_0, \balp) = o(1)$ in the case that $|\balp|$ is of `intermediate' size. The success of our endeavour depends crucially on our irrationality hypothesis for $\bL$. 

In order for the argument to work, we need to essentially replace $F$ by a function $\cF$ defined on $\bR^r$. For $\balp \in \bR^r$, let $\cF(\balp; P)$ be the supremum of the quantity
\[
\prod_{v \le n} (q + P|q \lam_v - a_v|)^{-1}
\]
over $q \in \bN$ and $\ba \in \bZ^n$, where  $\blam = \Lam \balp$. Note that if $(\alp_0,\balp) \in \fR_P$ then
\begin{equation} \label{Fs}
F(\alp_0, \balp; P) \le \cF(\balp; P).
\end{equation}
Moreover, since $\Lam$ has full rank, we have
\begin{equation} \label{Lam}
|\balp| \ll |\Lam \balp| \ll |\balp|.
\end{equation}

\begin{lemma} \label{BGF1}
Let $0 < V \le W$. Then
\begin{equation} \label{BGF1eq}
\sup_{V \le |\Lam \balp| \le W} \cF(\balp; P) \to 0 \qquad (P \to \infty).
\end{equation}
\end{lemma}

\begin{proof}
Suppose for a contradiction that \eqref{BGF1eq} is false. Then there exist $\psi > 0$ and
\[
(\balp^{(m)}, P_m, q_m, \ba^{(m)}) \in \bR^r \times [1,\infty) \times \bN \times \bZ^n \qquad (m \in \bN)
\]
such that the sequence $(P_m)$ increases monotonically to infinity, and such that if $m \in \bN$ then
\begin{enumerate}[(i)]
\item
\[
V \le |\blam^{(m)}| \le W
\]
and 
\item
\begin{equation} \label{BGFc2}
\prod_{v \le n} (q_m + P_m|q_m \lam_v^{(m)} - a^{(m)}_v|) < \psi^{-1},
\end{equation}
\end{enumerate}
where $\blam^{(m)} = \Lam \balp^{(m)}$ ($m \in \bN$). Now $q_m < \psi^{-1} \ll 1$, so $|\ba^{(m)}| \ll 1$. In particular, there are only finitely many possible choices for $(q_m, \ba^{(m)})$, so this pair must take a particular value infinitely often, say $(q, \ba)$. 

From \eqref{BGFc2}, we see that $q \blam^{(m)}$ converges to $\ba$ on a subsequence. The sequence $(|\balp^{(m)}|)_m$ is bounded so, by compactness, we know that $\balp^{(m)}$ converges to some vector $\balp$ on a subsubsequence. Therefore $q \Lam \balp = \ba$, and in particular $\Lam \balp$ is a rational vector, so $\balp \cdot \bL$ is a rational form. Note that $\balp \ne \bzero$, since $|\balp^{(m)}| \gg 1$. This contradicts our hypothesis on $\bL$, thereby establishing \eqref{BGF1eq}.
\end{proof}

\begin{cor} \label{BGF2}
Let $\tet$ be a small positive real number. Then there exists a function $T: [1,\infty) \to [1,\infty)$, increasing monotonically to infinity, such that $T(P) \le P^\tet$ and
\begin{equation} \label{FreemanBound}
\sup_{P^{\tet - 1} \le |\Lam \balp| \le T(P)} \cF(\balp; P) \le T(P)^{-1}.
\end{equation}
\end{cor}

\begin{proof} Lemma \ref{BGF1} yields a sequence $(P_m)$ of large positive real numbers such that
\[
\sup_{1/m \le |\Lam \balp| \le m} \cF(\balp; P_m) \le 1/m.
\]
We may assume that this sequence is increasing, and that $P_m^{\tet} \ge m$ ($m \in \bN$). Define $T(P)$ by $T(P) = 1$ ($1 \le P \le P_1$) and $T(P) = m$ ($P_m \le P < P_{m+1}$). Note that $T(P) \le P^\tet$, and that $T(P)$ increases monotonically to infinity. Now
\[
\sup_{T(P)^{-1} \le |\Lam \balp| \le T(P)} \cF(\balp; P) \le T(P)^{-1},
\]
for if $P \ge P_m$ then $\cF(\balp; P) \le \cF(\balp; P_m)$.

The inequality \eqref{FreemanBound} plainly holds if $P \le P_1$. Thus, it remains to show that if $P$ is large and
\begin{equation} \label{rest}
|\Lam \balp| < T(P)^{-1} < \cF(\balp; P)
\end{equation}
then $|\Lam \balp| < P^{\tet - 1}$. Suppose we have \eqref{rest}, with $P$ large. Writing $\blam = \Lam \balp$, we have
\[
\prod_{v \le n} (q + P|q \lam_v - a_v|) < T(P)
\]
for some $q \in \bN$ and some $\ba \in \bZ^n$. Now $q < T(P)^{1/n}$ and
\[
|q \blam - \ba| < T(P) / P,
\]
so the triangle inequality and \eqref{rest} give
\[
|\ba| < T(P)/P + T(P)^{1/n - 1} < 1.
\]
Therefore $\ba = \bzero$, and so
\[
|\Lam \balp| \le |q \blam| < T(P) / P \le P^{\tet - 1},
\]
completing the proof.
\end{proof}

Let $C_3$ be a large positive constant. Let $T(P)$ be as in Corollary \ref{BGF2}, with $\tet = \del^2 \tet_0$. We define our Davenport--Heilbronn major arc by
\[
\fM = \{ (\alp_0, \balp) \in \bR \times \bR^r: |\balp| < C_3 P^{\del^2 \tet_0 - 1} \},
\]
our minor arcs by
\[
\fm = \{ (\alp_0, \balp) \in \bR \times \bR^r: C_ 3 P^{\del^2 \tet_0 - 1} \le |\balp| \le C_3^{-1} T(P) \}
\]
and our trivial arcs by
\[
\ft = \{ (\alp_0, \balp) \in \bR \times \bR^r: |\balp| > C_3^{-1} T(P) \}.
\]
It follows from \eqref{Fs}, \eqref{Lam} and \eqref{FreemanBound} that
\begin{equation} \label{Freeman}
\sup_{\fR \cap \fm} F(\alp_0,\balp) \le T(P)^{-1}.
\end{equation}

Let $\fU$ be an arbitrary unit hypercube in $r$ dimensions, and put $\cU = U \times \fU$. By \eqref{MV} and \eqref{Freeman}, we have
\[
\int_{\fR \cap \fm \cap \cU} |S_0(\alp_0,\balp)|  \d \alp_0 \d \balp \ll T(P)^{-\eps} P^{n-r-3}.
\]
Now \eqref{Ldef}, \eqref{Kbounds} and \eqref{Kprod} yield
\begin{equation} \label{MinorBound}
\int_{\fR \cap \fm} |S_0(\alp_0,\balp) \bK_\pm (\balp)|  \d \alp_0 \d \balp = o(P^{n-r-3}).
\end{equation}
Note that
\begin{equation} \label{Ftrivial}
0 < F(\alp_0,\balp) \le 1.
\end{equation}
Together with \eqref{Ldef}, \eqref{Kbounds}, \eqref{Kprod} and \eqref{MV}, this gives
\begin{align} \notag 
\int_{\fR \cap \ft} |S_0(\alp_0,\balp) \bK_\pm (\balp)|  \d \alp_0 \d \balp 
&\ll P^{n-r-3} L(P)^r \sum_{n=0}^\infty (C_3^{-1} T(P) + n)^{-2} \\
\label{TrivialBound}
&\ll P^{n-r-3} L(P)^r T(P)^{-1} = o(P^{n-r-3}).
\end{align}
Coupling \eqref{MinorBound} with \eqref{TrivialBound} yields
\[
\int_{\fR \setminus \fM} |S_0(\alp_0,\balp) \bK_\pm (\balp)|  \d \alp_0 \d \balp = o(P^{n-r-3}).
\]

Recall that to show \eqref{asymp} it remains to establish \eqref{goal4}. Defining
\[
S_1 = \int_{\fR \cap \fM} S_0(\alp_0, \balp) e(-\balp \cdot \btau) \bK_\pm (\balp) \d \alp_0 \d \balp,
\]
it now suffices to prove that
\[
S_1 = (2\eta)^r \fS \chi_w P^{n - r - 3} + o(P^{n-r-3}).
\]
By \eqref{Kdef}, we have
\[
K_{\pm}(\alp) = (2 \eta \pm \rho) \cdot \sinc(\pi \alp \rho) \cdot \sinc(\pi \alp(2\eta \pm \rho)).
\]
Now \eqref{Tbound}, \eqref{Ldef} and the Taylor expansion of $\sinc(\cdot)$ yield
\[
K_{\pm}(\alp) = 2 \eta + O(L(P)^{-1}) \qquad (|\alp| < P^{-1/2}).
\]
Substituting this into \eqref{Kprod} shows that if $(\alp_0,\balp) \in \fM$ then
\begin{equation} \label{SincTaylor}
\bK_\pm(\balp) = (2\eta)^r + O(L(P)^{-1}).
\end{equation}
Moreover, it follows from \eqref{MV} and \eqref{Ftrivial} that
\begin{equation} \label{UpperBound}
\int_{\fR \cap\fM} |S_0(\alp_0,\balp)| \d \alp_0 \d \balp \ll P^{n-r-3}.
\end{equation}
From \eqref{SincTaylor} and \eqref{UpperBound}, we infer that
\[
S_1 = (2\eta)^r \int_{\fR \cap \fM} S_0(\alp_0,\balp) e(-\balp \cdot \btau)  \d \alp_0 \d \balp + o(P^{n-r-3}).
\]

Thus, to prove \eqref{asymp}, it remains to show that
\begin{equation} \label{S2goal}
S_2 = \fS \chi_w P^{n - r - 3} + o(P^{n-r-3}),
\end{equation}
where 
\[
S_2 = \int_{\fR \cap \fM} S_0(\alp_0,\balp) e(-\balp \cdot \btau) \d \alp_0 \d \balp.
\]
For $q \in \bN$ and $a \in \{ 1, 2, \ldots, q \}$, let $X(q,a)$ be the set of $(\alp_0,\balp) \in \bR \times \bR^r$ satisfying 
\[
q \le C_1 P^{2 \tet_0}, \qquad |q\alp_0 - a| < P^{2 \tet_0 - 3}.
\]

\begin{lemma} \label{trick}
Assume \eqref{R2b}. Then $\fR(q,a,\bzero) \cap \fM = X(q,a) \cap \fM$.
\end{lemma}

\begin{proof}
As $\fR(q,a, \bzero) \subseteq X(q,a)$, we have $\fR(q,a,\bzero) \cap \fM \subseteq X(q,a) \cap \fM$. Next, suppose that $(\alp_0,\balp) \in X(q,a) \cap \fM$. Then $|\balp| < C_3 P^{\del^2 \tet_0-1}$, so
\[
|\Lam \balp| \ll P^{\del ^2 \tet_0-1}.
\]
Now $|q \Lam \balp| \ll P^{(2+\del^2)\tet_0 - 1}$, and in particular we have \eqref{R1} with $\tet = \tet_0$ and $\ba = \bzero$. Thus, we have $(\alp_0,\balp) \in \fR(q,a,\bzero)$, and plainly $(\alp_0,\balp) \in\fM$.
\end{proof}

Note also that if $(\alp_0,\balp) \in \fR \cap \fM$ then $(\alp_0,\balp) \in \fR(q,a,\bzero)$ for some $q,a \in \bZ$ satisfying \eqref{R2b}. Indeed, if $(\alp_0,\balp) \in \fR(q,a,\ba) \cap \fM$ for some $q,a,\ba$ satisfying \eqref{R2b}, then the triangle inequality implies that $\ba = \bzero$. By Lemma \ref{trick}, we conclude that $\fR \cap \fM$ is the disjoint union of the sets $X(q,a) \cap \fM$, over $q,a \in \bZ$ satisfying \eqref{R2b}. Put
\[
V(q) = [-q^{-1} P^{2 \tet_0 -3}, q^{-1} P^{2 \tet_0 -3}] \qquad (q \in \bN)
\]
and
\[
W = [-C_3 P^{\del^2 \tet_0 - 1}, C_3 P^{\del^2 \tet_0 - 1}]^r.
\]
Now
\[
S_ 2 =  \sum_{q \le C_1 P^{2 \tet_0}} \sum_{\substack{a=1 \\ (a,q)= 1}}^q \displaystyle \int_{V(q) \times W} 
f_{q,a}(\bet_0,\balp) e(-\balp \cdot \btau) \d \bet_0 \d \balp,
\]
where
\begin{equation} \label{fdef}
f_{q,a}(\bet_0,\balp) = (P/q)^nS_{q,a,\bzero}I(P^3 \bet_0, P \Lam \balp).
\end{equation}

To prove \eqref{S2goal}, we complete the integrals and the outer sum to infinity. In light of \eqref{numvars}, it follows from Lemmas \ref{Sbound} and \ref{Ibound} that if $(a,q) = 1$ then
\begin{equation} \label{fbound}
f_{q,a}(\bet_0,\balp) \ll P^n q^{-3} (1 + P^3|\bet_0|)^{-1-\eps} \prod_{v \in V} (1+ P|\lam_v|)^{-1-\eps},
\end{equation}
where $V$ is as in \S \ref{MVE} and $\blam = \Lam \balp$. Let
\[
S_3 = \sum_{q \le C_1 P^{2 \tet_0}} \sum_{\substack{a=1 \\ (a,q)= 1}}^q \displaystyle \int_{\bR \times W} 
f_{q,a}(\bet_0,\balp) e(-\balp \cdot \btau) \d \bet_0 \d \balp.
\]
By \eqref{fbound} and an invertible change of variables, we have
\begin{align} \notag
S_2 - S_3 & \ll P^n \sum_{q \in \bN} q^{-3} \sum_{\substack{a=1 \\ (a,q)= 1}}^q 
\int_{q^{-1} P^{2 \tet_0 - 3}}^\infty
(P^3 \bet_0)^{-1-\eps} \d \bet_0 \\
\notag & \qquad \cdot
\int_{\bR^r} \prod_{v \in V} (1+P|\lam_v|)^{-1-\eps} \d \blam_V \\
\label{S2S3} &= o(P^{n-r-3}),
\end{align}
where $\blam_V = (\lam_v)_{v \in V}$.

Let
\[
S_4 = \sum_{q \le C_1 P^{2 \tet_0}} \sum_{\substack{a=1 \\ (a,q)= 1}}^q \displaystyle \int_{\bR^{r+1}} 
f_{q,a}(\bet_0,\balp) e(-\balp \cdot \btau) \d \bet_0 \d \balp.
\]
By \eqref{fbound} and an invertible change of variables, we have
\begin{align*} 
S_3 - S_4 & \ll P^n \sum_{q \in \bN} q^{-3} \sum_{\substack{a=1 \\ (a,q)= 1}}^q 
\int_\bR
(1 + P^3 |\bet_0|)^{-1-\eps} \d \bet_0 \\
& \qquad \cdot
\int \prod_{j \le r} (1+P|\alp'_j|)^{-1-\eps} \d \balp'.
\end{align*}
Here the inner integral is over $\balp' \in \bR^r$ such that $ \Lam_V^{-1} \balp'\notin W$, where $\Lam_V$ is the submatrix of $\Lam$ determined by taking rows indexed by $V$. With $c$ a small positive constant, we now have
\begin{equation} \label{S3S4}
S_3 - S_4 \ll P^{n-r-2} \int_{cP^{\del^2 \tet_0 - 1}}^\infty (P \alp)^{-1-\eps} \d \alp = o(P^{n-r-3}).
\end{equation}

Let
\[
S_5 = \sum_{q \in \bN} \sum_{\substack{a=1 \\ (a,q)= 1}}^q \displaystyle \int_{\bR^{r+1}} 
f_{q,a}(\bet_0,\balp) e(-\balp \cdot \btau) \d \bet_0 \d \balp.
\]
By \eqref{fbound} and an invertible change of variables, we have
\begin{align} \notag
S_4 - S_5 & \ll P^n \sum_{q > C_1 P^{2 \tet_0}} q^{-3} \sum_{\substack{a=1 \\ (a,q)= 1}}^q 
\int_\bR
(1 + P^3 |\bet_0|)^{-1-\eps} \d \bet_0 \\
\notag & \qquad \cdot
\int_{\bR^r} \prod_{j \le r} (1+P|\alp'_j|)^{-1-\eps} \d \balp' \\
\label{S4S5}
& \ll P^{n-r-3} \sum_{q > C_1 P^{2 \tet_0}} q^{-2} = o(P^{n-r-3}).
\end{align}
In view of \eqref{SS}, \eqref{sumdef} and \eqref{fdef}, we have
\[
S_5 = P^n \fS \int_{\bR^{r+1}} e(-\balp \cdot \btau)I(P^3 \bet_0, P \Lam \balp) \d \bet_0 \d \balp.
\]
Changing variables yields
\begin{equation} \label{S5alt}
S_5 = P^{n-r-3} \fS \int_{\bR^{r+1}} e(- P^{-1} \balp \cdot \btau) I(\bet_0, \Lam \balp) \d \bet_0 \d \balp.
\end{equation}

By \eqref{chidef} and \eqref{intdef}, we have
\[
\chi_w = \int_{\bR^{r+1}} I(\bet_0, \Lam \balp) \d \bet_0 \d \balp.
\]
As $h \ge 17 + 8r$, the bounds \eqref{Iineq} and
\[
e(-P^{-1} \balp \cdot \btau) - 1  \ll P^{-1} |\balp|
\]
imply that 
\[
\int_{\bR^{r+1}} e(- P^{-1} \balp \cdot \btau) I(\bet_0, \Lam \balp) \d \bet_0 \d \balp
= \chi_w + O(P^{-1}).
\]
Substituting this into \eqref{S5alt} yields
\[
S_5 = P^{n-r-3} \fS \chi_w + o(P^{n-r-3}).
\]
Combining this with \eqref{S2S3}, \eqref{S3S4} and \eqref{S4S5} yields \eqref{S2goal}, completing the proof of \eqref{asymp}.

\section{Positivity of the singular series and singular integral}
\label{positivity}

In this section, we show that $\fS > 0$ and $\chi_w > 0$, thereby completing the proof of Theorem \ref{thm1}. Since $\fS$ is the singular series associated to the cubic form $C$, its positivity is already well understood, and we briefly explain this here. By \cite[Lemma 7.1]{Bir1962} and the discussion preceding it, it suffices to show that if $p$ is prime then $C$ has a nonsingular $p$-adic zero. 

Let $p$ be a prime number. A well known example due to Mordell \cite{Mor1937} demonstrates the existence of cubic forms in nine variables with no nontrivial $p$-adic zeros. However, since $C$ has more than nine variables, a theorem due to Demyanov \cite{Dem1950} and Lewis \cite{Lew1952} implies that $C(\ba) = 0$ for some $\ba \in \bQ_p \setminus \{ \bzero \}$. We shall use a degeneracy argument, also attributable to Lewis \cite[p. 50]{Lew1969}, to infer the existence of a nonsingular zero.

If $\ba$ is nonsingular then we are done, so assume that $\ba$ is singular, and extend to a basis $\{ \ba, \be_2, \ldots, \be_n \}$ for $\bQ_p^n$. Define a cubic form
\[
C^*(\by) = C(y_1 \ba + y_2 \be_2 + \ldots + y_n \be_n).
\]
The Taylor expansion yields
\[
C^*(\by) = y_1^3 C(\ba) + y_1 ^2 \sum_{j=2}^n y_j  \partial_j C^*(1, 0, \ldots, 0)
+ y_1 Q(y_2, \ldots, y_n) + C^\dag(y_2, \ldots, y_n),
\]
where $Q$ is a quadratic form and $C^\dag$ is a cubic form. As $\ba$ is a singular zero of $C$, the gradient vector of $C$ at $\ba$ vanishes with respect to any basis. In particular, we have
\[
\partial_j C^*(1, 0, \ldots, 0) = 0 \qquad (1 \le j \le n).
\]
Thus,
\[
C^*(\by) = y_1 Q(y_2, \ldots, y_n) + C^\dag(y_2, \ldots, y_n).
\]

If $Q$ is not identically zero then we may choose $y_2, \ldots, y_n \in \bQ_p$ such that $Q(y_2, \ldots, y_n) \ne 0$, and choose
\[
y_1 = - \frac{ C^\dag(y_2, \ldots, y_n) } {Q(y_2, \ldots, y_n)};
\]
this gives a nonsingular zero $(y_1, y_2, \ldots, y_n)$ of $C^*$. Since $C^*$ is obtained from $C$ by an invertible change of variables, we would then have a nonsingular zero of $C$.

Suppose instead that $Q$ is identically zero. Then $C^*(\by) = C^\dag(y_2, \ldots, y_n)$; in other words, $C$ is \emph{degenerate} (see \cite[p. 68]{Dav2005}). Repeating the argument over and over leads to a nonsingular zero or a contradiction, since $h(C) \ge 25$ (note that the $h$-invariant is unaffected by changes of basis). We conclude that $\fS > 0$.

For positivity of the singular integral, we begin by establishing the equivalence of the definitions \eqref{chidef} and \eqref{cSchmidt}. Lemma \ref{Ibound} provides the appropriate analogy to \cite[Lemma 11]{Sch1982b}. Thus, following \cite[ch. 11]{Sch1982b} shows that the two definitions are equivalent.

We now work with the definition \eqref{cSchmidt}. We claim that $I_L(\bf) \gg 1$. Since $w(\bx) \gg 1$ for $\bx \in B := \ \bx \in \bR^n: \{ |\bx| \le 1/2 \}$, it suffices to show that
\begin{equation} \label{intbox}
\int_B \Psi_L(\bf(\bx)) \d \bx \gg 1.
\end{equation}
Define a real manifold
\[
\cM = \{ C = L_1 = \ldots = L_r = 0 \} \subseteq \bR^n.
\]
All of our forms have odd degree, so $\cM \cap A \ne \{ \bzero \}$ for every $(r+2)$-dimensional subspace $A$ of $\bR^n$. Thus, by \cite[Lemma 1]{Sch1982a}, we have $\dim(\cM) \ge n-r-1$. The argument of \cite[Lemma 2]{Sch1982b} now confirms \eqref{intbox}, thereby establishing the positivity of $\chi_w$. This completes the proof of Theorem \ref{thm1}.

\section{A more general result}
\label{general}

In this section we prove Theorem \ref{thm2}. We begin by establishing that  if $\balp \in \bR^r \setminus \{ \bzero \}$ then $\balp \cdot \bL$ is not a rational form. Suppose that $\balp \cdot \bL$ is a rational form, for some $\balp \in \bR^r$. Then $\Lam \balp = \bq$ for some $\bq \in \bQ^n$, where $\Lam$ is given by \eqref{LamDef}. Note that $\Lam$ has full rank, since its entries are algebraically independent over $\bQ$ and its $r \times r$ minors are nontrivial integer polynomials in these entries. It therefore follows from $\Lam \balp = \bq$ that $\alp_1, \ldots, \alp_r$ are rational functions in the entries of $\Lam_V$ over $\bQ$, where $V$ is as in \S \ref{MVE} and $\Lam_V$ is the submatrix of $\Lam$ determined by taking rows indexed by $V$. Let $i \in \{ 1,2,\ldots,n \} \setminus V$, and consider the equation
\begin{equation} \label{consider}
\alp_1 \lam_{1,i} + \alp_r \lam_{r,i} = q_i.
\end{equation}
Since $\alp_1, \ldots, \alp_r$ are rational functions in the entries of $\Lam_V$ over $\bQ$, equation \eqref{consider} and the algebraic independence of the entries of $\Lam$ necessitate that $\balp = \bzero$. We conclude that  if $\balp \in \bR^r \setminus \{ \bzero \}$ then $\balp \cdot \bL$ is not a rational form.

Thus, by Theorem \ref{thm1}, we may assume that $h \le 16 + 8r$, and so $n-h > r$. Write
\[
C = A_1 B_1 + \ldots + A_h B_h,
\]
where $A_1, \ldots, A_h$ are rational linear forms and $B_1, \ldots, B_h$ are rational quadratic forms. The vector space defined by
\[
A_1 = \ldots = A_h = 0
\]
has a rational subspace of dimension $n-h$, by the rank-nullity theorem. Let $\bz_1, \ldots, \bz_{n-h}$ be linearly independent integer points in this subspace. Define
\[
L'_i(\by) = L_i (y_1 \bz_1 + \ldots + y_{n-h} \bz_{n-h}) \qquad (1 \le i \le r).
\]
We seek to show that $\bL'(\bZ^{n-h})$ is dense in $\bR^r$. Writing $\bz_j = (z_{j,1}, \ldots, z_{j,n})$ ($1 \le j \le n-h$), and recalling \eqref{lamdef}, we have
\[
L'_i(\by) = \sum_{j \le n-h} \lam'_{i,j} y_j,
\]
where
\[
\lam'_{i,j} = \sum_{k \le n} \lam_{i,k} z_{j,k} \qquad (1 \le i \le r, \quad 1 \le j \le n-h).
\]

\begin{lemma} The $\lam'_{i,j}$ are algebraically independent over $\bQ$.
\end{lemma}

\begin{proof}
Extend $\bz_1, \ldots, \bz_{n-h}$ to a basis $\bz_1, \ldots, \bz_n$ for $\bQ^n$, and define
\[
\lam'_{i,j} = \sum_{k \le n} \lam_{i,k} z_{j,k} \qquad (1 \le i \le r, \quad n-h < j \le n).
\]
We now have an invertible rational matrix
\[
Z = \begin{pmatrix} z_{1,1} & \ldots & z_{1,n} \\
\vdots & & \vdots \\
z_{n,1} & \ldots & z_{n,n},
\end{pmatrix}.
\]
where $\bz_j = (z_{j,1}, \ldots, z_{j,n})$ ($1 \le j \le n$). Put
\[
\Lam' = \begin{pmatrix} \lam'_{1,1} & \ldots & \lam'_{r,1} \\
\vdots & & \vdots \\
\lam'_{1,n} & \ldots & \lam'_{r,n}
\end{pmatrix},
\]
and note that $\Lam' = Z \Lam$.

We shall prove, \emph{a fortiori}, that the entries of $\Lam'$ are algebraically independent over $\bQ$. Let $P'$ be a rational polynomial in $rn$ variables such that $P'(\Lam') = 0$. Define a rational polynomial $P$ in $rn$ variables by
\[
P(\Xi) = P'(Z \Xi), \qquad 
\Xi \in \mathrm{Mat}_{n \times r}.
\]
Now
\[
P(\Lam) = P'(Z\Lam) = P'(\Lam') = 0,
\]
so the algebraic independence of the entries of $\Lam$ forces $P$ to be the zero polynomial. Since
\[
P'(\Xi) = P(Z^{-1} \Xi)
\]
identically, the polynomial $P'$ must also be trivial.
\end{proof}

Thus, the entries of the matrix
\[
A = \begin{pmatrix}
\lam'_{1,1} &\ldots & \lam'_{1,n-h} \\
\vdots && \vdots \\
\lam'_{r,1} &\ldots & \lam'_{r,n-h}
\end{pmatrix}
\]
are algebraically independent over $\bQ$, and we seek to show that 
\[
\{ A \bx : \bx \in \bZ^{n-h} \}
\]
is dense in $\bR^r$. We put $A$ in the form $(I | \Lam'')$, where $I$ is the $r \times r$ identity matrix and $\Lam''$ is an $r \times (n-h-r)$ matrix, by the following operations:

\begin{enumerate}[(i)]
\item Divide the top row by $A_{11}$, so that now $A_{11} = 1$.
\item Subtract multiples of the top row from other rows so that 
\[
A_{21} = \ldots = A_{r1} = 0.
\]
\item Proceed similarly for columns $2,3, \ldots, r$.
\end{enumerate}

It suffices to show that the image of $\bZ^{n-h}$ under left multiplication by $(I | \Lam'')$ is dense in $\bR^r$.

\begin{lemma}
The entries of $\Lam''$ are algebraically independent over $\bQ$.
\end{lemma}

\begin{proof} After step (i), the entries of $A$ other than the top-left entry are algebraically independent. Indeed, suppose
\[
P \Bigl (\frac{A_{12}}{A_{11}}, \ldots, \frac{A_{1,n-h}}{A_{11}}, (A_{ij})_{\substack{2 \le i \le r \\ 1 \le j \le n-h}} \Bigr) = 0
\]
for some polynomial $P$ with rational coefficients, where the $A_{ij}$ are the entries of $A$ prior to step (i). For some $t \in \bN$, we can multiply the left hand side by $A_{11}^t$ to obtain a polynomial $P^*$ in $(A_{ij})_{\substack{1 \le i \le r \\ 1 \le j \le n-h}}$ with rational coefficients. The algebraic independence of the $A_{ij}$ implies that $P^*$ is the zero polynomial, and so $P$ must also be the zero polynomial.

After step (ii), the entries of $A$ excluding the first column are algebraically independent. Indeed, suppose 
\[
P \Bigl(A_{12}, \ldots, A_{1,n-h}, (A_{ij} - A_{i1} A_{1j})_{\substack {2 \le i \le r \\ 2 \le j \le n-h}} \Bigr) = 0
\]
for some polynomial $P$ with rational coefficients, where the $A_{ij}$ are the entries of $A$ prior to step (ii). The left hand side may be regarded as a polynomial $P^*$ in the $A_{ij}$ ($(i,j) \ne (1,1)$). The algebraic independence of the $A_{ij}$ ($(i,j) \ne (1,1)$) implies that $P^*$ is the zero polynomial, and so $P$ must also be the zero polynomial. 

We may now ignore column 1, and deal with columns $2,3,\ldots,r$ similarly.
\end{proof}

Next, consider the forms $L''_1, \ldots, L''_r$ given by
\[
L''_i(\bx) = \mu_{i,1} x_1 + \ldots + \mu_{i,n-h-r} x_{n-h-r} \qquad (1 \le i \le r),
\]
where $\mu_{i,j} = \Lam''_{ij}$ ($1 \le i \le r$, $1 \le j \le n-h-r$). It remains to show that $\bL''(\bZ^{n-h-r})$ is dense modulo 1 in $\bR^r$. We shall in fact establish equidistribution modulo 1 of the values of $\bL''(\bN^{n-h-r})$.

For this we use a multidimensional Weyl criterion. With $m=n-h-r$, we need to show that if $\bh \in \bZ^r \setminus \{\bzero\}$ then
\[
P^{-m} \sum_{x_1, \ldots, x_m \le P} e( \bh \cdot \bL''(\bx)) \to 0
\]
as $P \to \infty$. The summation equals
\[
\prod_{j \le m} \sum_{x_j \le P} e \Bigl( x_j \sum_{i \le r} h_i \mu_{i,j} \Bigr),
\]
so it suffices to show that
\[
\sum_{i \le r} h_i \mu_{i,1} \notin \bQ.
\]
This follows from the algebraic independence of the $\mu_{i,j}$, so we have completed the proof of Theorem \ref{thm2}.

\section{Equidistribution}
\label{equidistribution}

In this section, we prove Theorem \ref{thm3}. Let $\bk$ be a fixed nonzero integer vector in $r$ variables. By a multidimensional Weyl criterion, we need to show that
\[
N_u(P)^{-1} \sum_{\substack{|\bx| \le P \\ C(\bx) = 0}} e(\bk \cdot \bL(\bx)) \to 0
\]
as $P \to \infty$, where
\[
N_u(P) = \# \{\bx \in \bZ^n: |\bx| \le P, \: C(\bx) = 0 \}.
\]
It is known that $P^{n-3} \ll N_u(P) \ll P^{n-3}$; see remark (B) in the introduction of \cite{Sch1985}. Thus, it remains to show that
\begin{equation} \label{WeylGoal}
\sum_{\substack{|\bx| < P \\ C(\bx) = 0}} e(\bk \cdot \bL(\bx)) = o(P^{n-3}).
\end{equation}

Let $\tet_0$ be a small positive real number, and let $U$ be as in \eqref{interval}. By rescaling, we may assume that $C$ has integer coefficients. By \eqref{orth}, the left hand side of \eqref{WeylGoal} is equal to
\[
\int_U S_u(\alp_0, \bk) \d \alp_0,
\]
where $S_u(\cdot, \cdot)$ is as defined in the introduction. Recall \eqref{lamdef} and \eqref{LamDef}. Note that
\[
S_u(\alp_0,\bk) = g_u(\alp_0, \blam^*),
\]
where $g_u(\cdot, \cdot)$ is as defined in \eqref{gudef} and $\blam^* = \Lam \bk \in \bR^n$ is fixed.

For $q \in \bN$ and $a \in \bZ$, let $\fN'(q,a)$ be the set of $\alp_0 \in U$ such that
\[
|q \alp_0 - a| < P^{2 \tet_0 - 3}.
\]
Recall that $C_1$ is a large positive real number. For positive integers $q \le C_1P^{2\tet_0}$, let $\fN'(q)$ be the disjoint union of the sets $\fN'(q,a)$ over integers $a$ that are relatively prime to $q$. Let $\fN'$ be the disjoint union of the sets $\fN'(q)$. By Lemma \ref{DL} and the classical pruning argument in \cite[Lemma 15.1]{Dav2005}, it now suffices to prove that
\[
\int_{\fN'} S_u(\alp_0, \bk) \d \alp_0 = o(P^{n-3}).
\]

Let $\alp_0 \in \fN'(q,a)$, with $q \le C_1P^{2\tet_0}$ and $(a,q) = 1$. Then
\[
S_u(\alp_0,\bk) = \sum_{\by \mmod q} e_q(aC(\by))  S_\by(q, \bet_0, \blam^*),
\]
where $\bet_0 = \alp_0 - a/q$, and where in general we define
\[
S_\by(q, \bet_0, \blam) = \sum_{\bz: \: |\by+q\bz| < P} e(\bet_0 C(\by + q \bz) + \blam \cdot (\by + q \bz)).
\]
Note that $|q \bet_0| < P^{2 \tet_0 - 3}$. Let $\fQ$ denote the set of positive integers $q \le C_1P^{2\tet_0}$ such that 
\[
\| q \lam^*_v \| < P^{7 \tet_0 - 1} \qquad (1 \le v \le n),
\]
and put
\[
\fQ' = \{ q \in \bN: q \le C_1P^{2\tet_0} \} \setminus \fQ.
\]

Suppose $q \in \fQ'$, and let $j$ be such that $\| q \lam^*_j \| \ge P^{7 \tet_0 - 1}$. To bound $S_\by(q,\bet_0,\blam^*)$ we reorder the summation, if necessary, so that the sum over $z_j$ is on the inside. We bound this inner sum using the Kusmin--Landau inequality \cite[Theorem 2.1]{GK1991}, and then bound the remaining sums trivially. Note that, as a function of $z_j$, the phase
\[
\bet_0 C(\by + q \bz) + \blam^* \cdot (\by + q \bz)
\]
has derivative 
\[
\bet_0  \frac \partial {\partial z_j} C(\by + q \bz) +  q \lam_j^*,
\]
which is monotonic in at most two stretches. As $\| q \lam^*_j \| \ge P^{7 \tet_0 - 1}$ and
\[
\bet_0  \frac \partial {\partial z_j} C(\by + q \bz) \ll  P^{2 \tet_0 - 1}
\]
over the range of summation, the Kusmin--Landau inequality tells us that the sum over $z_j$ is $O(P^{1-7\tet_0})$. The remaining sums are over ranges of length $O(P/q)$, so
\[
S_\by(q, \bet_0, \blam^*) \ll (P/q)^{n-1} P^{1-7\tet_0}.
\]
Therefore
\[
S_u(\alp_0,\bk) \ll qP^{n-7\tet_0}.
\]
Since $\meas(\fN'(q)) \ll P^{2 \tet_0 - 3}$, we now have
\[
\sum_{q \in \fQ'} \int_{\fN'(q)} S_u(\alp_0,\bk) \d \alp_0 \ll \sum_{q \le C_1P^{2 \tet_0}} P^{2\tet_0-3} qP^{n-7\tet_0}
= o(P^{n-3}).
\]
It therefore remains to show that
\[
\sum_{q \in \fQ} \int_{\fN'(q)} S_u(\alp_0, \bk) \d \alp_0 = o(P^{n-3}).
\]

We shall need to study the more general exponential sums $g_u(\alp_0,\blam)$. Let $q \in \bN$ with $q \le P$, and let $a, a_1, \ldots, a_n \in \bZ$. Put \eqref{put}, and write
\[
g_u^*(\alp_0, \blam) = (P/q)^n S_{q,a,\ba} I_u(P^3 \bet_0, P \bbet),
\]
where $S_{q,a,\ba}$ is given by \eqref{sumdef}, and where
\[
I_u(\gam_0, \bgam) = \int_{[-1,1]^n} e(\gam_0 C(\bx) + \bgam \cdot \bx) \d \bx.
\]

\begin{lemma} We have
\begin{equation} \label{ggstar}
g_u(\alp_0,\blam) - g_u^*(\alp_0,\blam) \ll qP^{n-1} (1+P^3|\bet_0| + P|\bbet|).
\end{equation}
\end{lemma}

\begin{proof}
First observe that
\begin{equation} \label{SuDecomp}
g_u(\alp_0,\blam) = \sum_{\by \mmod q} e_q(aC(\by)+ \ba \cdot \by) S_\by(q,\bet_0,\bbet),
\end{equation}
and that
\[
S_\by(q,\bet_0,\bbet) = \sum_{\substack{|\bx| < P \\ \bx \equiv \by \mmod q}} e(\bet_0 C(\bx) + \bbet \cdot \bx).
\]
By \cite[Lemma 8.1]{Bro2009}, we now have
\begin{align*}
S_\by(q,\bet_0,\bbet) &= q^{-n} \int_{[-P,P]^n} e(\bet_0 C(\bx) + \bbet \cdot \bx) \d \bx \\
& \qquad + O\Bigl( \frac{P^{n-1} (1+P^3|\bet_0| + P|\bbet|)} {q^{n-1}}\Bigr) \\
&= (P/q)^n I_u(P^3\bet_0, P\bbet) + O\Bigl( \frac{P^{n-1} (1+P^3|\bet_0| + P|\bbet|)} {q^{n-1}}\Bigr).
\end{align*}
Substituting this into \eqref{SuDecomp} yields \eqref{ggstar}.
\end{proof}

Suppose $\alp_0 \in \fN'(q,a)$ with $q \in \fQ$. With $\blam = \blam^*$ and $a_v$ the nearest integer to $q \lam_v$ ($1 \le v \le n$), put \eqref{put} and $S_u^*(\alp_0, \bk) = g_u^*(\alp_0, \blam^*)$. In light of the inequalities
\[
1 \le q \le C_1 P^{2 \tet_0}, \qquad |q \bet_0| < P^{2 \tet_0 - 3}, \qquad |q \bbet| < P^{7 \tet_0 - 1},
\]
the error bound \eqref{ggstar} implies that
\[
S_u(\alp_0,\bk) - S_u^*(\alp_0,\bk) \ll P^{n-1+7\tet_0}.
\]
Since
\[
\meas \bigl( \cup_{q \in \fQ} \fN'(q) \bigr) \ll P^{4 \tet_0 - 3},
\]
it now suffices to prove that
\begin{equation} \label{finalgoal}
\sum_{q \in \fQ} \int_{\fN'(q)} S_u^*(\alp_0, \bk) \d \alp_0 = o(P^{n-3}).
\end{equation}

The final ingredient that we need for a satisfactory mean value estimate is an unweighted analogue of \eqref{Iineq}.

\begin{lemma} We have
\begin{equation} \label{Iuineq}
I_u(\gam_0, \bgam) \ll
\frac1{1+|\gam_0|^{h/8-\eps} + |\bgam|^{1/3}}.
\end{equation}
\end{lemma}

\begin{proof}
Since $I_u(\gam_0,\bgam) \ll 1$, we may assume that $|\gam_0| + |\bgam|$ is large. Specialising $(q,a,\ba) = (1,0,\bzero)$ in \eqref{ggstar}, it follows that
\[
I_u(\gam_0,\bgam) = P^{-n}g_u(\gam_0/P^3, \bgam/P) + O(P^{-1} (|\gam_0| + |\bgam|)).
\]
Since $I_u(\gam_0,\bgam)$ is independent of $P$, we are free to choose $P = (|\gam_0| + |\bgam|)^n$. By Lemma \ref{gbound}, with $\xi = \eps/n^2$, we now have
\[
I_u(\gam_0,\bgam) \ll P^{\eps/n^2} |\gam_0|^{-h/8} + \frac{ |\gam_0| + |\bgam|}P
\ll P^{\eps/n^2} |\gam_0|^{-h/8},
\]
so
\begin{equation} \label{IuPrelim}
I_u(\gam_0,\bgam) \ll \frac{(|\gam_0| + |\bgam|)^{\eps/n}}{|\gam_0|^{h/8}}.
\end{equation}

Let $C_4$ be a large positive constant. As $|\gam_0| + |\bgam|$ is large and $h \ge 17$, the desired inequality follows from \eqref{IuPrelim} if $|\bgam| < C_4 |\gam_0|$. Thus, we may assume that $|\bgam| \ge C_4 |\gam_0|$. Choose $j \in \{1,2,\ldots,n \}$ such that $|\bgam| = |\gam_j|$. Observe that
\[
I_u(\gam_0,\bgam) \ll \sup_{-1 \le x_i \le 1 \: (i \ne j)} \Biggl| \int_{-1}^1 e(\gam_0 C(\bx) + \gam_j x_j) \d x_j \Biggr|.
\]
As $|\gam_j| \ge C_4 |\gam_0|$, the bound \cite[Theorem 7.3]{Vau1997} now implies that
\[
I_u(\gam_0,\bgam) \ll |\gam_j|^{-1/3} = |\bgam|^{-1/3}.
\]
Combining this with \eqref{IuPrelim} gives
\[
I_u(\gam_0,\bgam) \ll 
\frac{(|\gam_0| + |\bgam|)^{\eps/n}}{|\gam_0|^{h/8} + |\bgam|^{1/3}(|\gam_0| + |\bgam|)^{\eps/n}}
\ll \frac{|\bgam|^{\eps/n}}{|\gam_0|^{h/8} + |\bgam|^{1/3+\eps/n}}.
\]
Considering cases and recalling that $h \le n$, we now have
\[
I_u(\gam_0,\bgam) \ll 
\frac1{|\gam_0|^{h/8-\eps} + |\bgam|^{1/3}}.
\]
This delivers the sought estimate \eqref{Iuineq}, since $|\gam_0| + |\bgam| \gg 1$.
\end{proof}

Let $\alp_0 \in \fN'(q,a)$, with $q \in \fQ$ and $(a,q) = 1$. The inequalities \eqref{Sineq}, \eqref{Iuineq} and $h \ge 17$ give
\[
S_u^*(\alp_0, \bk) \ll P^n q^{-2-\eps} (1 + P^3|\alp_0 - a/q|)^{-1-\eps} F(\bk; q, P)^\eps,
\]
where
\[
 F(\bk; q, P) = \prod_{v \le n} (q + P \| q \lam^*_v \|)^{-1}.
\]
Therefore
\begin{align} \notag
\sum_{q \in \fQ} \int_{\fN'(q)} |S_u^*(\alp_0, \bk)| F(\bk; q,P)^{-\eps} \d \alp_0
&\ll P^n \sum_{q \in \bN} q^{-1-\eps} \int_\bR (1+P^3 |\bet_0|)^{-1-\eps} \d \bet_0
\\
\label{finalMV} &\ll P^{n-3}.
\end{align}

As $\bk$ is a fixed nonzero vector, we have $1 \ll |\bk| \ll 1$. In particular, by \eqref{Lam}, we have $1 \ll |\Lam \bk| \ll 1$. Thus, Corollary \ref{BGF2} gives
\[
F(\bk; q, P) \le \cF(\bk, P) = o(1)
\]
as $P \to \infty$. Coupling this with \eqref{finalMV} yields \eqref{finalgoal}, completing the proof of Theorem \ref{thm3}.

\section{The singular locus}
\label{nonsingular}

The \emph{singular locus} of $C$ is the complex variety cut out by vanishing of $\nabla C$. Let $\cS$ be the singular locus of $C$, and let $\sig$ be the affine dimension of $\cS$. Let $h$ be the $h$-invariant of $C$, and let $A_1, \ldots, A_h, B_1, \ldots, B_h$ be as in \eqref{hdef}. Then $\cS$ contains the variety $\{ A_1 = \ldots = A_h = B_1 = \ldots B_h = 0 \}$, and so $\sig \ge n - 2h$. In particular, the conclusions of Theorem \ref{thm1} are valid if the hypothesis \eqref{numvars} is replaced by the condition
\[
n - \sig > 32 + 16 r.
\]
Thus, these conclusions hold for any nonsingular cubic form in more than $32 + 16 r$ variables.

However, one could improve upon this using a direct approach. Note that $h$ could be replaced by $n - \sig$ in Lemma \ref{DL}; the resulting lemma would be almost identical to \cite[Lemma 4.3]{Bir1962}, and again the weights and lower order terms are of no significance. The remainder of the analysis would be identical, and lead us to conclude that Theorem \ref{thm1} is valid with $h$ replaced by $n - \sig$. We could even use the same argument for positivity of the singular series, for if $r \ge 1$ then
\[
h \ge \frac{n-\sig}2 > \frac{16+8r}2 \ge 12 > 9.
\]
In particular, the conclusions of Theorem \ref{thm1} would hold for any nonsingular cubic form in more than $16 + 8 r$ variables. Similarly, Theorem \ref{thm3} is valid with $h$ replaced by $n - \sig$, and so its conclusion would hold for any nonsingular cubic form in more than sixteen variables.

Finally, we challenge the reader to improve upon these statements using more sophisticated technology, for instance to reduce the number of variables needed to solve the system \eqref{system}. It is likely that van der Corput differencing could be profitably incorporated, similarly to \cite{HB2007}. One might also hope to do better by assuming that $C$ is nonsingular, as in \cite{HB1983}.

\bibliographystyle{amsbracket}
\providecommand{\bysame}{\leavevmode\hbox to3em{\hrulefill}\thinspace}

\end{document}